\documentclass[12pt,a4paper]{amsart}
\makeatletter
\@namedef{subjclassname@1991}{2020 Mathematics Subject Classification}
\makeatother

\usepackage{amssymb}
\usepackage{hyperref}
\usepackage{amsfonts}
\usepackage{amsthm}
\usepackage{amsmath}
\usepackage{amscd}
\usepackage[latin2]{inputenc}
\usepackage{t1enc}
\usepackage[mathscr]{eucal}
\usepackage{indentfirst}
\usepackage{graphicx}
\usepackage{graphics}
\usepackage{pict2e}
\usepackage{epic}
\numberwithin{equation}{section}
\usepackage[margin=2.9cm]{geometry}
\usepackage{epstopdf}

\theoremstyle{plain}
\newtheorem{thm}{Theorem}[section]
\newtheorem{lem}[thm]{Lemma}
\newtheorem{cor}[thm]{Corollary}
\newtheorem{prop}[thm]{Proposition}

\theoremstyle{definition}
\newtheorem{defn}[thm]{Definition}

\newtheorem{rem}[thm]{Remark}
\newtheorem{?}[thm]{Problem}

\theoremstyle{definition}
\newtheorem*{nt*}{Notation}

\theoremstyle{plain}

\newcommand{\re}{{\rm Re}\,}

\newcommand{\inner}[2]{\left\langle {#1}, {#2} \right\rangle}

\newcommand{\calk}{\mathcal K}
\newcommand{\calc}{\mathcal C}

\newcommand {\C} {\mathbb C}
\newcommand {\R} {\mathbb R}

\newcommand {\bft} {\mathbf{t}}

\newcommand {\tha} {\theta}
\newcommand {\lbd} {\lambda}
\newcommand {\vap} {\varphi}
\newcommand {\al} {\alpha}
\newcommand {\be} {\beta}

\newcommand {\spc} {\mathcal{F}^2(\mathbb{C}^d)}

\newcommand {\cj} {\mathcal{C}_{A,b,c}}

\begin{document}
\title[Complex symmetry in the Fock space]{Complex symmetry in the Fock space of several variables}

\author{Pham Viet Hai}%
\address[P.V. Hai]{Faculty of Mathematics and Informatics, Hanoi University of Science and Technology, Khoa Toan-Tin, Dai hoc Bach khoa Hanoi, 1 Dai Co Viet, Hanoi, Vietnam.}%
%\email{phamviethai86@gmail.com}%%
\email{hai.phamviet@hust.edu.vn}

\author{Pham Trong Tien}
\address[P.T. Tien]{Faculty of Mathematics, Mechanics and Informatics, VNU University of Science, Vietnam National University, Hanoi}%
\email{phamtien@vnu.edu.vn}%%



\begin{abstract}
In this paper we study the complex symmetry in the several variable Fock space by using the techniques of weighted composition operators and semigroups. We characterize unbounded weighted composition operators that are (real) complex symmetric with respect to a concrete conjugation. Using this characterization, we study complex symmetric semigroups and their generators. We realize such generators as first-order differential operators.
\end{abstract}

\maketitle

\section{Introduction}
Let $\C^d$ be the $d$-dimensional  complex Euclidean space. The \emph{Fock space} $\spc$ consists of all entire functions on $\C^d$ for which
$$
\|f\| := \left(\frac{1}{\pi^d}\int_{\C^d}|f(z)|^2 e^{-|z|^2}\;dV(z)\right)^{1/2} < \infty,
$$
where $dV$ is the Lebesgue measure on $\C^d$. It is a reproducing kernel Hilbert space with the inner product
$$
\displaystyle\langle f,g\rangle := \frac{1}{\pi^d}\int_{\C^d} f(z)\overline{g(z)}e^{-|z|^2}\;dV(z)
$$
and \emph{kernel functions} $K_z(u):=e^{\langle u,z \rangle}, u, z \in \C^d$. 

For a nonzero entire function $\psi$ on $\C^d$ and a holomorphic self-mapping $\varphi$ of $\C^d$, the formal \emph{weighted composition expression} is defined by
$$
E_{\psi,\vap}(f) :=\psi\cdot f\circ\vap.
$$
Bounded and compact weighted composition operators on $\spc$ induced by the expression $E_{\psi,\vap}$ were firstly  characterized by Ueki \cite{zbMATH05131441} in terms of a certain integral transform, which is quite difficult to use. Later, Le \cite{zbMATH06324457} gave much easier characterizations for the boundedness and compactness of these operators on $\mathcal F^2(\C)$. Recently, Khoi and Tien \cite{zbMATH07261550} extended Le's results to the operators acting between Fock spaces. For bounded weighted composition operators, normality and isometry on the Fock spaces $\mathcal F^2(\C)$ and $\spc$ were investigated by Le in \cite{zbMATH06324457} and, respectively, by Zhao in \cite{zbMATH06674963, zbMATH06787249}; and the other problems such as essential norm, compact differences, topological structure were studied in \cite{zbMATH07686498, zbMATH07438387, zbMATH07261550}.

In this paper we are interested in unbounded weighted composition operators arising from the expression $E_{\psi,\vap}$ on the Fock space $\spc$. 
For the reader's convenience, we present some notations of unbounded operators.  For an operator $T$, the symbol $\text{dom}(T)$ stands for the domain and $T^*$ means the adjoint. In particular, for a matrix $M$, the adjoint $M^*=\overline{M}^{\bft}$, where  $M^{\bft}$ is the transpose matrix and $\overline{M}$ is the conjugate matrix. The inclusion $T \preceq S$ is understood that $\text{dom}(T)\subseteq\text{dom}(S)$ and $T(f)=S(f)$ for $f\in\text{dom}(T)$. Writing $TS$ is meant an operator that acts on each $f\in\text{dom}(TS)$ by the rule $TS(f) :=T(S(f))$, where the domain $\text{dom}(TS)=\{f\in\text{dom}(S):Sf\in\text{dom}(T)\}.$ 
The \emph{maximal weighted composition operator} on $\spc$ corresponding to $E_{\psi,\vap}$ is defined as follows:
\begin{gather*}
    \text{dom}(C_{\psi,\vap,\max}):=\left\{f\in\spc:E_{\psi,\vap}(f)\in\spc\right\},\\
    C_{\psi,\vap,\max}(f):=E_{\psi,\vap}(f), \ f \in \text{dom}(C_{\psi,\vap,\max}).
\end{gather*}
The operator $C_{\psi,\vap}$ is called a \emph{weighted composition operator} on $\spc$ if $C_{\psi,\vap}\preceq C_{\psi,\vap,\max}$. The operator $C_{\psi,\vap}$ is called \emph{bounded} if $\text{dom}(C_{\psi,\vap})=\spc$ and the ratio of the norm of $C_{\psi,\vap}(f)$ to that of $f$ is bounded by the same number, over all nonzero functions $f$ in $\spc$. By the Closed Graph Theorem, the operator $C_{\psi,\vap,\max}$ is bounded on $\spc$ if $\text{dom}(C_{\psi,\vap,\max})=\spc$, and in this case, by \cite[Proposition 3.1]{zbMATH07261550}, the symbol $\varphi$ must be an affine mapping, more precisely, $\varphi(z) = Qz + q$ with a $d \times d$ matrix $Q$ such that $\|Q\| \leq 1$ and a vector $q$ in $\C^d$.

Recently, there has been a surge of interest in the problem of classifying weighted composition operators that are complex symmetric, posed by Garcia and Hammond \cite{zbMATH06454968}. The problem is studied in various spaces: \cite{zbMATH06320823, zbMATH07053055} treat for Hardy space, \cite{zbMATH07164446} is for Dirichlet spaces, \cite{zbMATH07239348} for Bergman spaces, \cite{zbMATH07190521, zbMATH07592496} for Newton space and finally \cite{zbMATH07216720, zbMATH06857361} for Fock space. 

The aim of this paper is to investigate the complex symmetry in the Fock space $\spc$ by using techniques of weighted composition operators and semigroups. In details, we extend the results in \cite{zbMATH07216720} for unbounded weighted composition operators in the one variable Fock space and in \cite{zbMATH06857361} for bounded weighted composition operators on the several variable Fock space. Moreover, we also apply the obtained results to study the symmetric property of a semigroup of such unbounded operators. 

The first question is to characterize real symmetric (i.e. selfadjoint) and skew-real symmetric  weighted composition operators. Recall that a linear operator $S$ on a Hilbert space $\mathcal H$ is called \emph{real symmetric} (or, \emph{skew-real symmetric}) if $S=S^*$ (or, $S = -S^*$, respectively).
\begin{thm}\label{thm-rs}
Let $\psi$ be a nonzero entire function on $\C^d$ and $\varphi$ a holomorphic self-mapping of $\C^d$.
\begin{itemize}
\item[(a)] The operator $C_{\psi,\vap,\max}$ is real symmetric (or, skew-real symmetric) if and only if the symbols $\psi$ and $\vap$ are of the forms
\begin{gather}\label{eq-rs}
     \psi(z)=\tha e^{\inner{z}{q}}, \ \vap(z)=Qz+q, 
\end{gather}
with 
\begin{equation}\label{eq-cd-rs}
Q=Q^*, \ \ \
q\in\C^d, \ \ \ \tha\in\R \ (\text{or}, \tha\in i\R, \text{ respectively}).
\end{equation}
\item[(b)] The operator $C_{\psi,\vap}$ is (skew-)real symmetric if and only if $C_{\psi,\vap} = C_{\psi,\vap,\max}$ and 
 the symbols $\psi$ and $\varphi$ are of the forms \eqref{eq-rs} with the conditions \eqref{eq-cd-rs}.
\end{itemize}
\end{thm}

The second question is devoted to the complex symmetry of weighted composition operators on $\spc$ with respect to a certain class of conjugations. Recall that an anti-linear mapping $\mathcal C$ on a Hilbert space $\mathcal H$ is called a \textit{conjugation} if it is both involutive and isometric; and a closed, densely defined linear operator $S: \text{dom}(S) \subset \mathcal H \to \mathcal H$ is called \textit{$\mathcal C$-symmetric} if $S \preceq \mathcal C S^* \mathcal C$ and \textit{$\mathcal C$-selfadjoint} if $S = \mathcal C S^* \mathcal C$.

For a $d \times d$ matrix $A$, a  vector $b$ in $\C^d$ and a complex number $c$, we consider the operator
\begin{gather*}
    \cj(f)(z):=ce^{\inner{z}{\overline{b}}}\overline{f\left(\overline{Az+b}\right)}.
\end{gather*}
By \cite[Theorem 3.13]{zbMATH06857361}, the operator $\cj$ represents a conjugation on $\spc$ if and only if $A, b, c$ satisfy the following conditions:
\begin{gather}\label{eq-cj}
\begin{cases}
A^{-1}=A^*=\overline{A},\\
A\overline{b}+b=0,\\
|c|^2e^{|b|^2}=1.   
\end{cases}
\end{gather}

 The description of $\cj$-selfadjoint weighted composition operators on $\spc$ is given as follows.

\begin{thm}\label{thm-cs}
Let $\cj$ be a conjugation on $\spc$, $\psi$ a nonzero entire function on $\C^d$, and $\varphi$ a holomorphic self-mapping of $\C^d$.
\begin{itemize}
\item[(a)] The operator $C_{\psi,\vap,\max}$ is $\cj$-selfadjoint if and only if the symbols $\psi$ and $\varphi$ are of the forms 
\begin{gather}\label{eq-cs}
\psi(z)=\tha e^{\inner{z}{\ell}},\vap(z)=Qz+q,    
\end{gather}
where
\begin{equation}\label{eq-cd-cs}
(AQ)^{\bft}=AQ, \ \ \ \ell=\overline{Aq}+\overline{b}-Q^*\overline{b}, \ \ \ \tha\in\C, \ \ \ q\in\C^d.
\end{equation}
\item[(b)]
The operator $C_{\psi,\vap}$ is $\cj$-selfadjoint if and only if $C_{\psi,\vap}=C_{\psi,\vap,\max}$ and the symbols $\psi$ and $\varphi$ are of the forms \eqref{eq-cs} with the conditions \eqref{eq-cd-cs}.
\end{itemize}
\end{thm}

The proofs of Theorems \ref{thm-rs} and \ref{thm-cs} are given in Section \ref{sec-wcos}. In particular, we study little more general questions of characterizing two maximal weighted composition operators $C_{\psi,\vap, \max}$ and  $C_{\xi,\phi, \max}$ that satisfy:
\begin{itemize}
\item[(i)] $C_{\psi,\vap, \max}^*\preceq C_{\xi,\phi, \max}$ in Subsection \ref{subsec-rs};
\item[(ii)] $\cj C_{\psi,\vap, \max}^*\cj\preceq C_{\xi,\phi, \max}$ in Subsection \ref{subsec-cs}. 
\end{itemize}
The key step in the study of these questions is to show that the symbols $\varphi$ and $\phi$ must be affine, and hence, unbounded symmetric weighted composition operators $C_{\psi, \varphi, \max}$ must be induced by the symbol  $\varphi(z) = Qz + q$, where $Q$ is a $d \times d$ matrix and $q$ is a vector in $\C^d$. Then, using this and the unitary similarity of $d \times d$ matrices (see, for instance, \cite{zbMATH06125590}), we can deal with the operators induced by the symbol $\varphi(z) = Qz + q$ with a $d \times d$ diagonal matrix $Q$. In Subsection \ref{subsec-csrs}, we also show that the class of $\cj$-selfadjoint unbounded weighted composition operators $C_{\psi, \varphi}$ is quite large by containing the class of real symmetric unbounded operators $C_{\psi, \varphi}$ and the class of normal bounded operators $C_{\psi, \varphi}$.

The third question is to study $\cj$-selfadjoint semigroups of unbounded weighted composition operators on $\spc$. Using Theorem \ref{thm-cs}, we give a characterization of  such semigroups and their generators in Section \ref{sec-sg}.

\begin{thm}\label{thm-cssg}
Let $\cj$ be a conjugation on $\spc$, $(\psi_t)_{t \geq 0}$ a family  of nonzero entire functions  and $(\varphi_t)_{t \geq 0}$ a family  of nonconstant holomorphic self-mappings  of $\C^d$. Then
the family $(C_{\psi_t,\varphi_t})_{t \geq 0}$ is a $\cj$-selfadjoint semigroup if and only if the equality $C_{\psi_t,\varphi_t} = C_{\psi_t,\varphi_t,\max}$ holds for every $t\geq 0$ and the symbols are of the following forms
\begin{gather}\label{eq-csg}
\psi_t(z)=\tha_t e^{\inner{z}{\ell_t}},\quad\vap_t(z)=Q_tz+q_t, 
\end{gather}
with the coefficients of the forms
\begin{gather}
\label{eq-co-csg} 
Q_t=e^{t\Omega}, \quad \ q_t=\int\limits_0^te^{s\Omega}q_\diamond\;ds, \quad \ 
\ell_t=\int\limits_0^te^{s\Omega^*}\ell_\diamond\,ds, \quad \ 
\tha_t=e^{\tha_\diamond t+\int\limits_0^t\inner{q_s}{\ell_\diamond}\,ds},
\end{gather}
where $\Omega$ is a $d \times d$ matrix  and $q_\diamond,\ell_\diamond,\tha_\diamond$ are vectors in $\C^d$ satisfying two conditions
\begin{gather}
\label{eq-cd-csg}
(A\Omega)^{\bft} = A\Omega \qquad \text{ and } \qquad 
\Omega^*\ell_\diamond=\overline{A\Omega q_\diamond}-(\Omega^*)^2\overline{b}.
\end{gather}
Moreover, in this case, the generator $G$ of the semigroup $(C_{\psi_t,\varphi_t})_{t \geq 0}$ is also $\cj$-symmetric and it is defined by \begin{gather*}
    \text{dom}(G)=\bigcup_{\omega\in\R}\{f\in\Sigma_\omega:(\tha_\diamond+\inner{z}{\ell_\diamond})f(z)+\inner{\nabla f(z)}{\Omega z+q_\diamond}\in\Sigma_\omega\},\\
    Gf(z)=(\tha_\diamond+\inner{z}{\ell_\diamond})f(z)+\inner{\nabla f(z)}{\Omega z+q_\diamond},
\end{gather*}
where the definition of the subspace $\Sigma_{\omega}$ is given in Section \ref{sec-sg}. 
\end{thm}

Some notations, terminologies, and basic facts of the semigroup theory and semigroups induced by weighted composition operators are stated in Subsection \ref{subsec-sg}. Based on these facts and Theorem \ref{thm-cs} we prove Theorem \ref{thm-cssg} in Subsection \ref{subsec-cssg}.

\section{Complex symmetric weighted composition operators}\label{sec-wcos}

In this section we give a characterization of (skew-)real symmetric and $\cj$-selfadjoint unbounded weighted composition operators on the Fock space $\spc$.
To do this, firstly we establish some auxiliary results.

\subsection{Preliminaries}
 Similarly to \cite[Lemma 3.1 and Proposition 3.4]{zbMATH07216720}, we have the following lemma.

\begin{lem}\label{W*Kz-prop}
Let $\psi$ be a nonzero entire function on $\C^d$, and $\varphi$ a holomorphic self-mapping of $\C^d$.
\begin{itemize}
\item[(a)] For each $z\in\C^d$, $K_z\in\text{dom}(C^*_{\psi,\vap})$ and  $C_{\psi,\vap}^*K_z=\overline{\psi(z)}K_{\vap(z)}$.
\item[(b)] The maximal weighted composition operator $C_{\psi,\vap,\max}$ is always closed.
\end{itemize}
\end{lem}

Given a pair of $d \times d$ unitary matrices $(U,V)$, a nonzero entire function $\psi$ on $\C^d$, and a holomorphic self-mapping $\varphi(z) = Qz + q$ of $\C^d$ with a $d \times d$ matrix $Q$ and a vector $q \in \C^d$, motivated by the approach in \cite{zbMATH07261550}, we define
$$
\widetilde{\psi}(z): = \psi(U^*z) \ \text{ and } \ \widetilde{\varphi}(z): = \widetilde{Q}z + \widetilde{q} \ \text{ where } \ \widetilde{Q}: = V^*QU^*, \ \widetilde{q}: = V^*q.
$$
Similarly to \cite[Proposition 3.4]{zbMATH07261550}, we obtain the following lemma.
\begin{lem} 
 It holds that
\begin{equation}\label{eq-as}
C_{\psi, \varphi, \max} = C_U C_{\widetilde{\psi}, \widetilde{\varphi}, \max}C_V \ \text{ and } \  C_{\widetilde{\psi}, \widetilde{\varphi}, \max} = C_{U^*} C_{\psi, \varphi, \max} C_{V^*}, 
\end{equation}
where, for a $d \times d$ matrix $Q$, the operator $C_Q$ is defined by $C_Q(f)(z): = f(Qz)$. 
\end{lem}

The next result extends \cite[Theorem 3.7]{zbMATH07216720} to the several variables case and plays an essential role in the current paper. 
 
\begin{prop}\label{prop202304250835}
Let 
\begin{gather*}
\begin{cases}
\psi(z)=\tha e^{\inner{z}{\overline{\lbd}}},\quad\vap(z)= Az + \beta,\\
\widehat{\psi}(z)=\overline{\tha}e^{\inner{z}{\be}},\quad\widehat{\vap}(z)=\overline{A}z+\overline{\lbd},
\end{cases}
\end{gather*}
where $A$ is a $d \times d$ diagonal matrix with diagonal elements $\al_1, ..., \al_d$, and $\beta, \lambda$ are two vectors in $\C^d$, and  $\theta$ is a number in $\C$. 
Then the following holds 
$$
C_{\psi,\vap,\max}^*=C_{\widehat{\psi},\widehat{\vap},\max}.
$$
\end{prop}
\begin{proof}
For each $z\in\C^d$, it can be checked that $E_{\psi,\vap}(K_z)=\tha e^{\inner{\be}{z}}K_{\widehat{\vap}(z)}\in\spc$. Hence,
\begin{gather*}
    K_z\in\text{dom}(C_{\psi,\vap,\max})\,\,\text{and}\,\,C_{\psi,\vap,\max}(K_z)=\tha e^{\inner{\be}{z}}K_{\widehat{\vap}(z)}.
\end{gather*}
First, we prove that
\begin{gather}\label{202304241334}
C_{\psi,\vap,\max}^*\preceq C_{\widehat{\psi},\widehat{\vap},\max}.    
\end{gather}
Indeed, for $f\in\text{dom}(C_{\psi,\vap,\max}^*)$ and $z\in\C^d$, we have $ C_{\psi,\vap,\max}^*(f) \in \spc$ and
\begin{align*}
C_{\psi,\vap,\max}^*(f)(z)&=\inner{C_{\psi,\vap,\max}^*(f)}{K_z}=\inner{f}{C_{\psi,\vap,\max}(K_z)}\\
&=\inner{f}{\tha e^{\inner{\be}{z}}K_{\widehat{\vap}(z)}}=\widehat{\psi}(z)\inner{f}{K_{\widehat{\vap}(z)}}=\widehat{\psi}(z)f(\widehat{\vap}(z))=E_{\widehat{\psi},\widehat{\vap}}(f)(z),
\end{align*}
which implies that \begin{gather*}
f\in\text{dom}(C_{\widehat{\psi},\widehat{\vap},\max}) \ \text{ and } \ 
C_{\widehat{\psi},\widehat{\vap},\max}(f)=E_{\widehat{\psi},\widehat{\vap}}(f)=C_{\psi,\vap,\max}^*(f).
\end{gather*}
Thus, \eqref{202304241334} holds.

Next is to prove the inverse inclusion of \eqref{202304241334}, that is equivalent to proving that
\begin{gather}\label{202304241344}
\inner{C_{\psi,\vap,\max}(f)}{g}=\inner{f}{C_{\widehat{\psi},\widehat{\vap},\max}(g)}\ \text{ for all } f\in\text{dom}(C_{\psi,\vap,\max}) \text{ and } g\in\text{dom}(C_{\widehat{\psi},\widehat{\vap},\max}).
\end{gather}
We use \cite[Lemma 2.1]{zbMATH07261550} and \cite[Theorem 3.7]{zbMATH07216720} to prove \eqref{202304241344} by reducing it to the case of one variable. To that aim, we express
\begin{align*}
& \inner{C_{\psi,\vap,\max}(f)}{g}\\
=&\; \frac{\tha}{\pi}\int_\C e^{\lbd_dz_d-|z_d|^2}dV(z_d)\cdots\frac{1}{\pi}\int_\C e^{\lbd_2z_2-|z_2|^2}dV(z_2)\frac{1}{\pi}\int_\C e^{\lbd_1z_1}f(\vap(z))\overline{g(z)}e^{-|z_1|^2}dV(z_1).
\end{align*}
Fix $z_2,\cdots,z_d$. By \cite[Lemma 2.1]{zbMATH07261550}, the functions
\begin{gather*}
z_1\mapsto f(z_1,z_2,\cdots,z_n),\ \  
z_1\mapsto g(z_1,z_2,\cdots,z_n),\ \  
z_1\mapsto e^{\lbd_1z_1}f(\vap(z))
\end{gather*}
belong to the Fock space $\mathcal{F}^2(\mathbb{C})$ of one variable  of $z_1$. Now use \cite[Theorem 3.7]{zbMATH07216720} to get
\begin{align*}
& \frac{1}{\pi}\int_\C e^{\lbd_1z_1}f(\vap(z))\overline{g(z)}e^{-|z_1|^2}dV(z_1)\\
= & \; \frac{1}{\pi}\int_\C f(z_1,\al_2z_2+\be_2,\cdots,\al_dz_d+\be_d)\overline{e^{\overline{\be_1}z_1} g(\overline{\al_1}z_1+\overline{\lbd_1},z_2,\cdots,z_d)}e^{-|z_1|^2}dV(z_1).
\end{align*}
Through substituting the equality above back into the inner product $\inner{C_{\psi,\vap,\max}(f)}{g}$ and then repeating the arguments to the remaining variables, we get \eqref{202304241344}.
\end{proof}

\subsection{(Skew-)Real symmetry} \label{subsec-rs}
In the subsection, we give the proof of Theorem \ref{thm-rs}. In fact, we do a little more by considering the problem of characterizing two maximal weighted composition operators $C_{\psi,\vap, \max}$ and  $C_{\xi,\phi, \max}$ that satisfy
\begin{gather*}
    C_{\psi,\vap, \max}^*\preceq C_{\xi,\phi, \max}.
\end{gather*}
The problem was given by Le \cite{zbMATH06074411} who considered bounded weighted composition operators on Hardy spaces. Then the authors \cite{zbMATH07562075} studied the problem for bounded composition operators on Hardy-Smirnov spaces.

\begin{lem}\label{lem202304242020}
Let $\psi, \xi$ be nonzero entire functions on $\C^d$ and $\varphi, \phi$ holomorphic self-mappings of $\C^d$. The following conditions are equivalent:
\begin{itemize}
    \item[(i)] $C^*_{\psi,\varphi, \max}K_z = C_{\xi,\phi, max}K_z$ for every $z \in \C^d$;
    \item[(ii)] $C^*_{\xi,\phi, \max}K_z = C_{\psi,\varphi, \max}K_z$ for every $z \in \C^d$;
    \item[(iii)] $\psi, \xi$ and $\varphi, \phi$ are of the forms \begin{gather}\label{202304241738abc}
\begin{cases}
\psi(z)=\tha e^{\inner{z}{p}},\vap(z)=Qz+q,\\
\xi(z)=\overline{\tha}e^{\inner{z}{q}},\phi(z)=Q^*z+p,
\end{cases}
\end{gather}    
where $Q$ is a $d \times d$ matrix, $p, q \in \C^d$ and $\theta \in \C$.
\end{itemize}
\end{lem}
\begin{proof}
It is sufficient to prove the equivalence of (i) and (iii). The one of (ii) and (iii) is similar.

Suppose that (i) holds. Using Lemma \ref{W*Kz-prop}, we can check that (i)
is equivalent to that
\begin{gather}\label{202304241534}
    \overline{\psi(z)}K_{\vap(z)}(x)=\xi(x)K_z(\phi(x))\ \text{ for all } \ x,z\in\C^d.
\end{gather}

Letting $z=0$ in \eqref{202304241534}, we obtain $\xi(x)=\overline{\psi(0)}K_{\vap(0)}(x)$, which gives the form of $\xi$ in \eqref{202304241738abc}, where $\tha=\psi(0)$ and $q=\vap(0).$ Now substituting the form back into \eqref{202304241534}, we get
\begin{gather}\label{202304241723}
\overline{\psi(z)}e^{\inner{x}{\vap(z)}}=\overline{\tha}e^{\inner{x}{q}+\inner{\phi(x)}{z}}.
\end{gather}
Throughout letting $x=0$ and then taking the conjugate of both sides of a complex equation, we get the form of $\psi$ in \eqref{202304241738abc}, where $p=\phi(0)$. Now \eqref{202304241723} is reduced to the following
\begin{gather*}
    \inner{p}{z}+\inner{x}{\vap(z)}=\inner{x}{q}+\inner{\phi(x)}{z}
\end{gather*}
or equivalently to saying that
\begin{gather}\label{202307101713}
\inner{x}{\vap(z)-q}=\inner{\phi(x)-p}{z}. 
\end{gather}
For each $j\in\{1,2,\cdots,d\}$, we differentiate \eqref{202307101713} with respect to the variable $x_j$ and then
\begin{gather*}
\overline{\vap_j(z)}-\overline{q_j}=\sum_{k=1}^d\frac{\partial\phi_k(x)}{\partial x_j}\overline{z_k}.
\end{gather*}
Hence, we get the form of $\vap$ in \eqref{202304241738abc}. Inserting $\vap$ into \eqref{202307101713}, we get the form of $\phi$ in \eqref{202304241738abc}.

Conversely, suppose that (iii) holds. For each $z, x \in \C^d$, by Lemma \ref{W*Kz-prop},
$$
C^*_{\psi, \varphi, \max}K_z(x) = \overline{\psi(z)} K_{\varphi(z)}(x) = \overline{\theta} e^{\langle p, z\rangle + \langle x, Qz + q \rangle} = 
 \overline{\theta} e^{\langle x,q \rangle + \langle Q^* x + p, z  \rangle} = C_{\xi, \phi, \max}K_z(x),
$$
which implies (i).
\end{proof}

\begin{prop}\label{202307121414}
Let $\psi, \xi$ be nonzero entire functions on $\C^d$ and $\varphi, \phi$ holomorphic self-mappings of $\C^d$.
Then $C_{\psi,\vap,\max}^*\preceq C_{\xi,\phi,\max}$ if and only if the symbols $\psi, \xi$ and $\varphi, \phi$ are of the forms \eqref{202304241738abc}.
If we assume additionally that 
$Q=Q^*$, then $C_{\psi,\vap,\max}^*=C_{\xi,\phi,\max}$.
\end{prop}
\begin{proof}
Let $C_{\psi,\vap,\max}^*\preceq C_{\xi,\phi,\max}$. Then, for each $z\in\C^d$, $K_z\in\text{dom}(C_{\psi,\vap,\max}^*)$ by Lemma \ref{W*Kz-prop}, hence
$$
K_z\in\text{dom}(C_{\xi,\phi,\max}) \ \text{ and } \ 
C_{\psi,\vap,\max}^*K_z=C_{\xi,\phi,\max}K_z,  
$$
which by Lemma \ref{lem202304242020} implies that the symbols $\psi, \xi$ and $\varphi, \phi$ are of the forms \eqref{202304241738abc}.

Conversely, let $\psi, \xi$ and $\varphi, \phi$ of the forms \eqref{202304241738abc}.
For each $f\in\text{dom}(C_{\psi,\vap,\max}^*)$, we have $C_{\psi,\vap,\max}^*(f) \in \spc$ and, by Lemmas \ref{W*Kz-prop} and \ref{lem202304242020},
\begin{align*}
C_{\psi,\vap,\max}^*(f)(z)&=\inner{C_{\psi,\vap,\max}^*(f)}{K_z}=\inner{f}{C_{\psi,\vap,\max} K_z} = \inner{f}{C^*_{\xi,\phi,\max}K_z}\\
&=\inner{f}{\overline{\xi(z)}K_{\phi(z)}}=\xi(z)\inner{f}{K_{\phi(z)}}=\xi(z)f(\phi(z))=E_{\xi,\phi}(f)(z),
\end{align*}
for every $z\in\C^d$, which implies that $f\in\text{dom}(C_{\xi,\phi,\max})$ and
\begin{gather*}
C_{\xi,\phi,\max}(f)=E_{\xi,\phi}(f)=C_{\psi,\vap,\max}^*(f).
\end{gather*}

Now we assume additionally that $Q=Q^*$. Then, by \cite[Theorem 4.1.5]{zbMATH06125590}, $Q=UMU^*$, where $U$ is a unitary matrix and $M$ is a real diagonal matrix. Then by \eqref{eq-as}, we get
$$
C_{\psi,\vap,\max} = C_{U^*}C_{\widetilde{\psi},\widetilde{\vap},\max}C_U,\quad C_{\xi,\phi,\max}=C_{U^*}C_{\widetilde{\xi},\widetilde{\phi},\max}C_U,
$$
where
\begin{gather*}
\begin{cases}
    \widetilde{\vap}(z)=Mz+U^*q, \ 
    \widetilde{\psi}(z)=\psi(Uz) = \tha e^{\inner{z}{U^*p}},\\
    \widetilde{\phi}(z)=Mz+U^*p,\ \widetilde{\xi}(z)=\xi(Uz)=\overline{\tha}e^{\inner{z}{U^*q}}.
\end{cases}
\end{gather*}
Since $M$ is a real diagonal matrix, Proposition \ref{prop202304250835} reveals that $C_{\widetilde{\psi},\widetilde{\vap},\max}^*=C_{\widetilde{\xi},\widetilde{\phi},\max}$. With all preparation in place, we can prove that $C_{\psi,\vap,\max}^*=C_{\xi,\phi,\max}$. Indeed, using \cite[Proposition 1.7]{zbMATH06046473}, we obtain
\begin{align*}
C_{\psi,\vap,\max}^*&=(C_{U^*}C_{\widetilde{\psi},\widetilde{\vap},\max}C_U)^*
=(C_{\widetilde{\psi},\widetilde{\vap},\max}C_U)^*C_U \\
& \succeq C_{U^*}C_{\widetilde{\psi},\widetilde{\vap},\max}^*C_U =C_{U^*}C_{\widetilde{\xi},\widetilde{\phi},\max}C_U=C_{\xi,\phi,\max}.
\end{align*}
The proof is completed.
\end{proof}

Now we use Proposition \ref{202307121414} to prove Theorem \ref{thm-rs}.
\begin{proof}[The proof of Theorem \ref{thm-rs}]
It is enough to prove for the real symmetry. The proof for the skew-real symmetry is similar.

The part (a) is an immediate consequence of Proposition \ref{202307121414} and the converse conclusion of the part (b) directly follows from the part (a). Now we suppose that the operator $C_{\psi,\varphi}$ is real symmetric. Firstly, we show that the operator $C_{\psi,\varphi,\max}$ is real symmetric. Indeed, since $C_{\psi,\varphi}\preceq C_{\psi,\varphi,\max}$, by \cite[Proposition 1.6]{zbMATH06046473}, we have
$$
C_{\psi,\varphi,\max}^*\preceq C_{\psi,\varphi}^*=C_{\psi,\varphi}\preceq C_{\psi,\varphi,\max}.
$$
By Proposition \ref{202307121414}, the symbols $\psi$ and $\varphi$ are of the forms \eqref{eq-rs} with the conditions \eqref{eq-cd-rs}, and hence, again by Proposition \ref{202307121414},
the operator $C_{\psi,\varphi,\max}$ is real symmetric. Then,
$$
C_{\psi,\varphi}\preceq C_{\psi,\varphi,\max}=C_{\psi,\varphi,\max}^*\preceq C_{\psi,\varphi}^*=C_{\psi,\varphi}, \text{ i.e. } C_{\psi,\varphi} = C_{\psi,\varphi,\max}.
$$    
The proof is completed.
\end{proof}

\begin{cor}
Let $\psi$ be a nonzero entire function on $\C^d$ and $\varphi$ a holomorphic self-mapping of $\C^d$ such that the domain of the operator $C_{\psi,\vap,\max}$ contains all kernel functions. Then the following statements are equivalent:
\begin{itemize}
\item[(i)] The operator $C_{\psi,\vap,\max}$ is real symmetric (or, skew-real symmetric);
\item[(ii)] $C_{\psi,\vap,\max}\preceq C_{\psi,\vap,\max}^*$ (or, $C_{\psi,\vap,\max}\preceq -C_{\psi,\vap,\max}^*$; respectively);
%The operator $C_{\psi,\vap,\max}$ is semi real symmetric (or, semi skew-real symmetric);
\item[(iii)] The symbols $\psi$ and $\vap$ are of the forms \eqref{eq-rs} with the conditions \eqref{eq-cd-rs}.
\end{itemize}
\end{cor}
\begin{proof}
The direction (i) $\Rightarrow$ (ii) is obvious; meanwhile, the direction (iii) $\Rightarrow$ (i) follows from Theorem \ref{thm-rs}. Moreover, if the operator $C_{\psi,\vap,\max}$ is real symmetric and its domain contains all kernel functions $K_z, z \in \C^d$, then from (ii) it follows that
$$
C_{\psi,\varphi,\max}K_z =  C_{\psi,\vap,\max}^* K_z \ \text{ for each } z\in\C^d,  
$$
which, by Lemma \ref{lem202304242020}, implies (iii).
\end{proof}

\subsection{Complex symmetry} \label{subsec-cs}
The aim of this subsection is to prove Theorem \ref{thm-cs}. We do a little more by characterizing two maximal weighted composition operators $C_{\psi,\vap, \max}$ and $C_{\xi,\phi, \max}$ that satisfy
$$
\cj C_{\psi,\vap, \max}^*\cj\preceq C_{\xi,\phi, \max}.     
$$
As it turns out, these operators must be generated by the symbols of the explicit forms. To do so, we need the following two lemmas.

\begin{lem}\label{lem-con}
    Always has $\cj K_z=ce^{\inner{b}{\overline{z}}}K_{A^*\overline{z}+\overline{b}}$ for every $z\in\C^d.$
\end{lem}
\begin{proof}
For every $z,x\in\C^d,$ we have
\begin{align*}
\cj K_z(x)&=ce^{\inner{x}{\overline{b}}}\overline{K_z\left(\overline{Ax+b}\right)}
=ce^{\inner{x}{\overline{b}}}e^{\inner{Ax+b}{\overline{z}}}\\
&=ce^{\inner{b}{\overline{z}}}e^{\inner{x}{A^*\overline{z}+\overline{b}}}
=ce^{\inner{b}{\overline{z}}}K_{A^*\overline{z}+\overline{b}}(x).
\end{align*}
\end{proof}

\begin{lem}\label{lem-202305071534}
Let $\cj$ be a conjugation on $\spc$, $\psi$ a nonzero entire function on $\C^d$ and $\vap$ a holomorphic self-mapping of $\C^d$. Then 
\begin{gather}\label{eq-ckz}
\cj C_{\psi,\vap,\max}^*\cj K_z=C_{\xi,\phi,\max}K_z \ \text{ for each } z\in\C^d, 
\end{gather}
if and only if the symbols are of the forms
\begin{gather}\label{202304261831}
\begin{cases}
\psi(z)=\tha e^{\inner{p-q}{\overline{b}}+\inner{z}{\delta}},\quad\vap(z)=Qz+q,\\
\xi(z)=\tha e^{\inner{z}{\ell}},\quad\phi(z)=A^*Q^{\bft}Az+p,
\end{cases}
\end{gather}
where $Q$ is a $d\times d$ matrix, $p, q \in \C^d$, $\theta \in \C$, and 
\begin{gather}\label{202307242121}
\begin{cases}
\delta=A^*\overline{p}+\overline{b}-Q^*\overline{b},\\
\ell=A^*\overline{q} +\overline{b}+A^*\overline{Q}b.
\end{cases} 
\end{gather}
\end{lem}
\begin{proof}
Let \eqref{eq-ckz} hold. Using Lemmas \ref{W*Kz-prop} and \ref{lem-con}, we can check that \eqref{eq-ckz} is equivalent to that
 \begin{gather}\label{202304261830}
    |c|^2 \psi(A^*\overline{z}+\overline{b}) e^{\inner{\overline{z}}{b}+\inner{\vap(A^*\overline{z}+\overline{b})}{\overline{b}}+\inner{x}{A^*\overline{\vap(A^*\overline{z}+\overline{b})}+\overline{b}}}=\xi(x)e^{\inner{\phi(x)}{z}} \ \text{for all } z,x\in\C^d.
\end{gather}
Setting $z=0$ in \eqref{202304261830}, we get the form 
\begin{equation*}
\xi(x)=\tha e^{\inner{x}{\ell}}, \ \text{ where } \ell=A^*\overline{\vap(\overline{b})}+\overline{b} \ \text{ and } \tha=|c|^2\psi(\overline{b})e^{\inner{\vap(\overline{b})}{\overline{b}}}.
\end{equation*}
Substituting the form back into \eqref{202304261830}, for every $z, x \in \C^d$, we obtain
\begin{gather}\label{202307240912}
|c|^2 \psi(A^*\overline{z}+\overline{b}) e^{\inner{\overline{z}}{b}+\inner{\vap(A^*\overline{z}+\overline{b})}{\overline{b}}+\inner{x}{A^*\overline{\vap(A^*\overline{z}+\overline{b})}+\overline{b}}}=\tha e^{\inner{x}{\ell}+\inner{\phi(x)}{z}}.   
\end{gather}
Letting $x=0$ in the line above, we get
\begin{gather}\label{202307240913}
|c|^2 \psi(A^*\overline{z}+\overline{b}) e^{\inner{\overline{z}}{b}+\inner{\vap(A^*\overline{z}+\overline{b})}{\overline{b}}}=\tha e^{\inner{\phi(0)}{z}}.    
\end{gather}
Substituting \eqref{202307240913} back into \eqref{202307240912} gives
\begin{gather*}
\inner{\phi(0)}{z}+\inner{x}{A^*\overline{\vap(A^*\overline{z}+\overline{b})}+\overline{b}}=\inner{x}{\ell}+\inner{\phi(x)}{z}
\end{gather*}
or equivalent to that
\begin{gather}\label{202305051338}
\inner{x}{\omega(z)}=\inner{\phi(x)-\phi(0)}{z},   \end{gather}
where
\begin{gather*}
\omega(z)=A^*\left(\overline{\vap(A^*\overline{z}+\overline{b})}-\overline{\vap(\overline{b})}\right)=(\omega_1(z),\omega_2(z),\cdots,\omega_d(z)).
\end{gather*}
For each $j \in \{1,2, \cdots, d\}$, differentiating \eqref{202305051338} with respect to the variable $\overline{z_j}$, we get
\begin{gather*}
\sum_{k=1}^d x_k \overline{\frac{\partial\omega_k(z)}{\partial z_j}}=\phi_j(x)-\phi_j(0).    
\end{gather*}
Consequently, the symbol $\phi$ takes the form $\phi(x)=Hx+p.$ Thus, by \eqref{202305051338}, the symbol $\vap$ takes the form $\vap(z)=Qz+q$. Now we insert these forms back into \eqref{202305051338}
\begin{gather*}
    \inner{x}{A^*\overline{QA^*}z}=\inner{Hx}{z},
\end{gather*}
which means $H^*=A^*\overline{QA^*}$, i.e., by \eqref{eq-cj}, $H = A^* Q^{\textbf{t}} A$.

Finally, putting $x: = A^* \overline{z} + \overline{b}$,  substituting it into \eqref{202307240913}, and using \eqref{eq-cj}, we get
\begin{align*}
\psi(x) &= \frac{\theta}{|c|^2} e^{\inner{x}{A^*\overline{p} - A^*b - Q^*\overline{b}} + \inner{A\overline{b}}{b - \overline{p}} - \inner{q} {\overline{b}}} \\
& = \frac{\theta}{|c|^2} e^{\inner{x}{\delta} - \inner{b}{b - \overline{p}} - \inner{q} {\overline{b}}} = \theta e^{\inner{p - q} {\overline{b}} + \inner{x}{\delta}}.
\end{align*}

Conversely, suppose that the symbols are of the forms \eqref{202304261831} with the conditions \eqref{202307242121}. Then, using \eqref{eq-cj}, we have
$$
\psi(A^* \overline{z} + \overline{b}) = \theta e^{\inner{p-q}{\overline{b}} + \inner{A^* \overline{z} + \overline{b}}{A^*\overline{p} +\overline{b} - Q^* \overline{b}}} 
 = \theta e^{\inner{A^* Q^{\textbf{t}}b - q + {\overline{b}}}{\overline{b}} + \inner{p - \overline{b} - A^* Q^{\textbf{t}}b}{z}}
$$
and
\begin{align*}
& \inner{\overline{z}}{b}+\inner{\vap(A^*\overline{z}+\overline{b})}{\overline{b}}+\inner{x}{A^*\overline{\vap(A^*\overline{z}+\overline{b})}+\overline{b}}
 \\
 = & \; \inner{A^* Q^{\textbf{t}}Ax + A^* Q^{\textbf{t}} b +  \overline{b}}{z} + \inner{q}{\overline{b}} + \inner{Q^{\textbf{t}} b}{b}  + \inner{x}{A^* \overline{Q}b + A^*\overline{q} + \overline{b}},
\end{align*}
which together with \eqref{eq-cj} implies that
\begin{align*}
 & |c|^2 \psi(A^*\overline{z}+\overline{b}) e^{\inner{\overline{z}}{b}+\inner{\vap(A^*\overline{z}+\overline{b})}{\overline{b}}+\inner{x}{A^*\overline{\vap(A^*\overline{z}+\overline{b})}+\overline{b}}} \\
 = & \; \theta e^{\inner{x}{ A^*\overline{q} + \overline{b} + A^* \overline{Q}b}} e^{\inner{A^* Q^{\textbf{t}}Ax + p}{z}} = \xi(x) e^{\inner{\phi(x)}{z}}.
\end{align*}
Thus, \eqref{202304261830}, and hence, \eqref{eq-ckz} hold.
\end{proof}

\begin{prop}\label{prop-202307261556}
Let $\cj$ be a conjugation on $\spc$, $\psi$ a nonzero entire function on $\C^d$ and $\vap$ a holomorphic self-mapping of $\C^d$. Then 
\begin{gather}\label{eq-cjcs}
\cj C_{\psi,\vap,\max}^*\cj\preceq C_{\xi,\phi,\max}
\end{gather}
if and only if the symbols $\psi, \xi$ and $\vap, \phi$ are of the forms \eqref{202304261831} with the conditions \eqref{202307242121}.
In this case, 
$$
\cj C_{\xi,\phi,\max}\cj=C_{\vartheta,\omega,\max}, \text{ where }
\begin{cases}
\vartheta(z)=\overline{\tha}e^{\inner{\overline{b}}{p-q}+\inner{z}{q}},\\
\omega(z)=Q^*z+\delta.
\end{cases}
$$
\end{prop}
\begin{proof}
Let \eqref{eq-cjcs} hold. By Lemmas \ref{W*Kz-prop} and \ref{lem-con}, we get
$$
K_z \in \text{dom}(\cj C_{\psi,\vap,\max}^*\cj), 
$$
which implies that 
$$
K_z \in \text{dom}( C_{\xi,\phi,\max}) \ \text{ and 
} \ \cj C_{\psi,\vap,\max}^*\cj K_z=C_{\xi,\phi,\max}K_z 
$$
for each $z \in \C^d$. From this and Lemma \ref{lem-202305071534} the assertion follows.

Conversely, suppose that the symbols $\psi, \xi$ and $\vap, \phi$ are of the forms \eqref{202304261831} with the conditions \eqref{202307242121}.
Note that the desired inclusion \eqref{eq-cjcs} is equivalent to that
\begin{gather}\label{202307261443}
C_{\psi,\vap,\max}^*\preceq \cj C_{\xi,\phi,\max}\cj. 
\end{gather}

First we claim that $\cj C_{\xi,\phi,\max}\cj=C_{\vartheta,\omega,\max}$.
Since the domains are maximal, the equality above can be reduced to proving that $\cj E_{\xi,\phi}\cj=E_{\vartheta,\omega}$. Indeed, for each $f\in\spc$, by the definition of $\cj$ and setting $H=A^*Q^{\bft}A$, we have 
\begin{gather*}
E_{\xi,\phi}\cj f(z)=\xi(z)\cj f(\phi(z))=\tha ce^{\inner{z}{\ell}+\inner{Hz+p}{\overline{b}}}\overline{f(\overline{AHz+Ap+b})},
\end{gather*}
hence,
\begin{align*}
\cj E_{\xi,\phi}\cj f(z)&=ce^{\inner{z}{\overline{b}}}\overline{E_{\xi,\phi}\cj f(\overline{Az+b})}\\
&=\overline{\tha}|c|^2e^{\inner{z}{\overline{b}}+\inner{Az+b}{\overline{\ell}}+\inner{\overline{H}Az+\overline{H}b+\overline{p}}{b}}f(\overline{AH}Az+\overline{AH}b+\overline{Ap+b})\\
&=\overline{\tha}|c|^2e^{\inner{z}{\overline{b}+A^*\overline{\ell}+A^*\overline{H}^*b}+\inner{b}{\overline{\ell}}+\inner{\overline{H}b+\overline{p}}{b}}f(\overline{AH}Az+\overline{AH}b+\overline{Ap+b}).
\end{align*}
Note that using \eqref{eq-cj}, we get
\begin{align*}
\overline{AH}Az+\overline{AH}b+\overline{Ap+b} &= Q^*z+Q^*\overline{A}b+\overline{Ap+b}\\
&=Q^*z-Q^*\overline{b}+\overline{Ap}-\overline{A}b\\
&=Q^*z-Q^*\overline{b}+A^*\overline{p} + \overline{b}
\end{align*}
and
\begin{align*}
\overline{b}+A^*\overline{\ell}+A^*\overline{H}^*b
&=\overline{b}+A^*(A^{\bft}Q\overline{b}+A^{\bft}q+b)+A^*\overline{H}^*b\\
& =\overline{b}+Q\overline{b}+q+A^*b+QA^*b \\
& =Q(\overline{b}+A^*b)+(\overline{b}+A^*b)+q=q
\end{align*}
and
\begin{align*}
\inner{b}{\overline{\ell}}+\inner{\overline{H}b+\overline{p}}{b}
&=\inner{b}{A^{\bft}Q\overline{b}+A^{\bft}q+b}+\inner{\overline{H}b+\overline{p}}{b}\\
&=\inner{b}{-A^{\bft}QA^*b+A^{\bft}q+b}+\inner{\overline{H}b+\overline{p}}{b}\\
&=\inner{b}{A^{\bft}q+b}+\inner{\overline{p}}{b}
=|b|^2+\inner{\overline{b}}{p-q}.
\end{align*}
From these and again \eqref{eq-cj}, we get the claim. 

Next, for every $z,x\in\C^d$, we observe
\begin{align}\label{eq-ekz}
\nonumber
E_{\psi,\vap}K_z(x) &=\tha e^{\inner{p-q}{\overline{b}} + \inner{x}{\delta}+\inner{Qx+q}{z}} \\
&=\tha e^{\inner{p-q}{\overline{b}} + \inner{q}{z}}e^{\inner{x}{Q^*z + \delta}} =\tha e^{\inner{p-q}{\overline{b}} + \inner{q}{z}}K_{Q^*z+\delta}(x)\in\spc,
\end{align}
which means that $K_z\in\text{dom}(C_{\psi,\vap,\max})$ for each $z \in \C^d$.

Finally, for every $f\in\text{dom}(C_{\psi,\vap,\max}^*)$ and $z\in\C^d$, using \eqref{eq-ekz} and the kernel functions, we have
\begin{align*}
C_{\psi,\vap,\max}^*f(z)&=\inner{C_{\psi,\vap,\max}^*f}{K_z}=\inner{f}{C_{\psi,\vap,\max}K_z}\\
&=\overline{\tha}e^{\inner{\overline{b}}{p-q}+\inner{z}{q}}f(Q^*z+\delta) = C_{\vartheta,\omega, \max} = \cj C_{\xi,\phi,\max}\cj f(z),
\end{align*}
which implies \eqref{202307261443}.
\end{proof}

\begin{proof}[The proof of Theorem \ref{thm-cs}]
(a) For the necessary condition, suppose that the operator $C_{\psi,\vap,\max}$ is $\cj$-selfadjoint, which, by Lemma \ref{lem-202305071534}, gives the desired result.

For the sufficient condition, suppose that $\psi$ and $\varphi$ are of the forms \eqref{eq-cs} with the conditions \eqref{eq-cd-cs}. Then, by Proposition \ref{prop-202307261556}, $C_{\psi,\vap,\max}^*\preceq \cj C_{\psi,\vap,\max}\cj$. Now it remains to prove the inverse inclusion. By \eqref{eq-cd-cs}, the matrix $AQ$ is symmetric, hence, by \cite[Corollary 2.6.6]{zbMATH06125590}, we can express $AQ=UMU^{\bft}$, where $U$ is unitary and $M$ is a real diagonal matrix. Use \eqref{eq-cj} to get $Q=A^*UMU^{\bft}$. 
Thus, by \eqref{eq-as}, we obtain
$$ C_{\psi,\vap,\max}=C_{U^{\bft}}C_{\widetilde{\psi},\widetilde{\vap},\max}C_{A^*}C_U
$$
where
\begin{gather*}
\begin{cases}
\widetilde{\psi}(z)=\psi((U^{\bft})^*z)=\tha e^{\inner{(U^{\bft})^*z}{\ell}}=\tha e^{\inner{z}{U^{\bft}\ell}},\\
\widetilde{\vap}(z)=Mz+U^*Aq.
\end{cases}
\end{gather*}
Proposition \ref{prop-202307261556} gives $\cj C_{\psi,\vap,\max}\cj=C_{\vartheta,\omega,\max}$, where
\begin{gather*}
\begin{cases}
\vartheta(z)=\overline{\tha}e^{\inner{z}{q}},\\
\omega(z)=Q^*z+\ell.
\end{cases}
\end{gather*}
Now we consider the operator $C_{\vartheta,\omega,\max}$. Since $Q=A^*UMU^{\bft}$ and $M$ is a real diagonal matrix, we infer $Q^*=(U^{\bft})^*MU^*A$.
This and \eqref{eq-as} imply that
\begin{gather*}
C_{\vartheta, \omega, \max}=  C_{U^*}C_AC_{\widetilde{\vartheta},\widetilde{\omega},\max}C_{(U^{\bft})^*} \ \text{ with } \
\begin{cases}
\widetilde{\vartheta}(z)=\overline{\tha}e^{\inner{z}{U^*Aq}},\\
\widetilde{\omega}(z)=Mz+U^{\bft}\ell.      
\end{cases}
\end{gather*}
Next, for $z\in\C^d$, by \eqref{eq-ekz}, we have 
$$
K_z\in\text{dom}(C_{\psi,\vap,\max}) \text{ and } 
C_{\psi,\vap,\max}K_z=\tha e^{\inner{q}{z}}K_{Q^*z+\ell}.
$$  

With all preparation in place, we prove the inclusion $ \cj C_{\psi,\vap,\max}\cj \preceq  C_{\psi,\vap,\max}^*$.
Using \cite[Proposition 1.7(ii)]{zbMATH06046473}, we get
\begin{align*}
C_{\psi,\vap,\max}^*&=(C_{U^{\bft}}C_{\widetilde{\psi},\widetilde{\vap},\max}C_{A^*}C_U)^*
=(C_{\widetilde{\psi},\widetilde{\vap},\max}C_{A^*}C_U)^*C_{(U^{\bft})^*} \\
& \succeq (C_{A^*}C_U)^*C_{\widetilde{\psi},\widetilde{\vap},\max}^*C_{(U^{\bft})^*} =C_{U^*}C_AC_{\widetilde{\psi},\widetilde{\vap},\max}^*C_{(U^{\bft})^*} \\
& = C_{U^*} C_A C_{\widetilde{\vartheta},\widetilde{\omega},\max}C_{(U^{\bft})^*} =C_{\vartheta,\omega,\max}=\cj C_{\psi,\vap,\max}\cj,
\end{align*}
where we use the fact $C_{\widetilde{\psi},\widetilde{\vap},\max}^* = C_{\widetilde{\vartheta},\widetilde{\omega},\max}$ by Proposition \ref{prop202304250835}.

(b) The sufficient condition follows from part (a).
For the necessary condition, suppose that the operator $C_{\psi,\varphi}$ is $\cj$-selfadjoint. We show that the operator $C_{\psi,\varphi,\max}$ is $\cj$-selfadjoint. Indeed, since $C_{\psi,\varphi}\preceq C_{\psi,\varphi,\max}$, by \cite[Proposition 1.6]{zbMATH06046473}, we have
$$
C_{\psi,\varphi,\max}^*\preceq C_{\psi,\varphi}^*=\cj C_{\psi,\varphi}\cj\preceq\cj C_{\psi,\varphi,\max}\cj,
$$
hence,
$$
\cj C_{\psi,\varphi,\max}^* \cj \preceq C_{\psi,\varphi,\max}.
$$
From this and Proposition \ref{prop-202307261556} it follows that $\psi$ and $\varphi$ are in the forms of \eqref{eq-cs} with the conditions \eqref{eq-cd-cs}. Hence, by part (a), the operator $C_{\psi,\varphi,\max}$ is $\cj$-selfadjoint. Thus, the desired conclusion follows from the following inclusions
$$
C_{\psi,\varphi}\preceq C_{\psi,\varphi,\max}=\cj C_{\psi,\varphi,\max}^*\cj\preceq\cj C_{\psi,\varphi}^*\cj=C_{\psi,\varphi}.
$$    
\end{proof}

For the $\cj$-symmetry of the maximal weighted composition operators, we have the following
\begin{cor}
Let $\cj$ be a conjugation on $\spc$, $\psi$ a nonzero entire function on $\C^d$, and $\varphi$ a holomorphic self-mapping of $\C^d$ such that the domain of the operator $C_{\psi,\vap,\max}$ contains all kernel functions. Then the following statements are equivalent:
\begin{itemize}
\item[(i)] The operator $C_{\psi,\vap,\max}$ is $\cj$-selfadjoint;
\item[(ii)] The operator $C_{\psi,\vap,\max}$ is $\cj$-symmetric; 
\item[(iii)] The symbols $\psi$ and $\varphi$ are of the forms \eqref{eq-cs} with the conditions \eqref{eq-cd-cs}.
\end{itemize}
\end{cor}
\begin{proof}
The direction (i) $\Rightarrow$ (ii) is obvious; meanwhile, the direction (iii) $\Rightarrow$ (i) follows from Theorem \ref{thm-cs}. Moreover, if the operator $C_{\psi,\vap,\max}$ is $\cj$-symmetric and its domain contains all kernel functions $K_z, z \in \C^d$, then
$$
C_{\psi,\varphi,\max}K_z = \cj C_{\psi,\vap,\max}^*\cj K_z \ \text{ for each } z\in\C^d,  
$$
which, by Lemma \ref{lem-202305071534}, implies (iii).
\end{proof}

\subsection{The particular cases of the $\cj$-symmetry} \label{subsec-csrs}
We begin the subsection with the fact that every (possibly unbounded) real symmetric operator is complex symmetric; in other words, there exists a conjugation $\calc$ such that the real symmetric operator is $\calc$-selfadjoint. In the context of weighted composition operators, we realize a such conjugation $\calc$ as an operator $\cj$ corresponding to an adapted and highly relevant choice of the parameters $A,b$ and $c.$

\begin{prop}\label{prop-20230927}
Let $\psi$ be a nonzero entire function on $\C^d$ and $\varphi$ a holomorphic self-mapping of $\C^d$.
 If the operator $C_{\psi,\vap}$ is (skew-)real symmetric, then it is $\cj$-selfadjoint for some conjugation $\cj.$
\end{prop}
\begin{proof}
By Theorem \ref{thm-rs}, the symbols $\psi$ and $\varphi$ are of the forms \eqref{eq-rs} with the conditions \eqref{eq-cd-rs}, i.e., 
$$
\psi(z)=\tha e^{\inner{z}{q}}, \ \vap(z)=Qz+q, 
$$
with $Q=Q^*, \ q\in\C^d$,  and $ \tha\in\R$ (or, $\tha\in i\R$, respectively).

To prove the existence of the conjugation $\cj$, we need more details of the matrix $Q$.
Let $\lambda_1, \dots, \lambda_k$ be the distinct eigenvalues with respective multiplicities $d_1,\dots, d_k$ of the matrix $Q$. Since $Q = Q^*$, by \cite[Theorem 2.5.6]{zbMATH06125590}, $\lambda_1, \dots, \lambda_k$ are real and there is a $d \times d$ unitary matrix $V$ such that 
\begin{equation}\label{eq-qm}
Q = V M V^*, \text{ where } M := \lambda_1 I_{d_1} \oplus \dots \oplus \lambda_k I_{d_k}. 
\end{equation}
Firstly we choose $b = 0, c=1$ and then find a matrix $A$ satisfying the conditions below:
\begin{itemize}
\item[(i)] $A^{-1} = A^* = \overline{A}$,
\item[(ii)] $(AQ)^{\textbf{t}} = AQ$,
\item[(iii)] $Aq = \overline{q}$.
\end{itemize}

Obviously, condition (i) means that $A$ is a unitary and symmetric matrix. Then, using \cite[Corollary 2.6.6]{zbMATH06125590} for the symmetric matrix $A$, we can find a unitary matrix $U$ and a nonnegative diagonal matrix $\Sigma$ such that $A = U \Sigma U^{\textbf{t}}$. If the matrix $A$ is unitary, then 
$$
I = AA^* = U \Sigma U^{\textbf{t}} (U^{\textbf{t}})^* M U^* = U \Sigma^2 U^*,
$$
which is equivalent to that  $\Sigma = I$.
Thus, $A$ satisfies condition (i) if and only if $A = U U^{\bft}$ with a unitary matrix $U$.

Now we consider condition (ii). 
We claim that a $d \times d$ unitary and symmetric matrix $A$ satisfies (ii)
if and only if $A$ is of the form
\begin{equation} \label{eq-ma}
A = \overline{V} \left(V_1V_1^{{\textbf{t}}} \oplus \cdots \oplus V_kV_k^{{\textbf{t}}}\right)V^*,
\end{equation}
where $V_1, \dots, V_k$ are $d_1 \times d_1$, $\dots$, $d_k \times d_k$ unitary matrices.
Indeed, for a such matrix $A$, condition (ii) is equivalent to that
$$
Q = (A^*\overline{V}) M (A^* \overline{V})^*. 
$$
By \cite[Theorem 2.5.4]{zbMATH06125590}, the last condition holds if and only if there are $d_1 \times d_1$ unitary matrix $W_1$, ..., $d_k \times d_k$ unitary matrix $W_k$ such that
$$
V = A^*\overline{V} \left(W_1 \oplus \cdots \oplus W_k \right), \text{ i.e., } A = \overline{V} \left(W_1 \oplus \cdots \oplus W_k \right)V^*.
$$
Moreover, since the matrix $A$ is unitary and symmetric, so are $W_1, \dots, W_k$, and hence, as above, $W_j = V_jV_j^{\textbf{t}}$ with a $d_j \times d_j$ unitary matrix $V_j$ for every $j = 1, \dots, k$. Thus, we get the claim.

Finally, we consider condition (iii). For a matrix of \eqref{eq-ma}, condition (iii) is equivalent to that
$$
 \left(V_1V_1^{{\textbf{t}}} \oplus \cdots \oplus V_kV_k^{{\textbf{t}}} \right)V^*q = \overline{V^*q}, 
$$
that is,
\begin{equation}\label{eq-vj}
V_j^{\textbf{t}} (V^*q)_{[d_j]} =  \overline{V_j^{\textbf{t}} (V^*q)_{[d_j]}}, \text{ i.e., } V_j^{\textbf{t}} (V^*q)_{[d_j]} \in \R^{d_j}, \ j = 1, \dots, k,
\end{equation}
where we divide the vector $V^*q$ into $k$ vectors $(V^*q)_{[d_j]}, \ j = 1, \dots, k$.

Consequently, the matrix $A$ satisfies conditions (i), (ii), (iii) if and only if it is of the form \eqref{eq-ma} with the condition \eqref{eq-vj}.
\end{proof}

By arguments similar to those used in Proposition \ref{prop-20230927}, the result below shows that the normality in the bounded sense is a particular case of the $\cj$-selfadjointness, too.
\begin{prop}
Let $\psi$ be a nonzero entire function on $\C^d$ and $\varphi$ a holomorphic self-mapping of $\C^d$.
If the operator $C_{\psi,\vap}$ is normal in the bounded sense, then it is $\cj$-selfadjoint for some conjugation $\cj.$
\end{prop}
\begin{proof}
By Theorem \cite{zbMATH06787249}, the symbols $\psi$ and $\varphi$ are of the forms
$$
\psi(z)=\tha e^{\inner{z}{\ell}}, \ \vap(z)=Qz+q, 
$$
where $Q$ is a $d \times d$ normal matrix with $\|Q\|\leq 1$ and $q,\ell $ are vectors in $\C^d$ such that 
$$
(I-Q)\ell=(I-Q^*)q \ \text{ and } \ \|\ell\|=\|q\|.
$$

Since $Q$ is a normal matrix, by \cite[Theorem 2.5.3]{zbMATH06125590}, $Q$ is of the form \eqref{eq-qm}, however in this case, the eigenvalues can be complex.

As in Proposition \ref{prop-20230927}, we choose $b = 0, c=1$ and then find a matrix $A$ satisfying the conditions below:
\begin{itemize}
\item[(i)] $A^{-1} = A^* = \overline{A}$,
\item[(ii)] $(AQ)^{\textbf{t}} = AQ$,
\item[(iii)] $(I - Q)\overline{Aq} = (I - Q^*)q$.
\end{itemize}
As above, the matrix $A$ is of the form \eqref{eq-ma}. Using this and \eqref{eq-qm}, condition (iii) is equivalent to that
$$
(1 - \overline{\lambda_j}) V_j^{\textbf{t}} (V^*q)_{[d_j]} =   \overline{(1 - \overline{\lambda_j}) V_j^{\textbf{t}} (V^*q)_{[d_j]}}, \text{ i.e., } (1 - \overline{\lambda_j}) V_j^{\textbf{t}} (V^*q)_{[d_j]} \in \R^{d_j}, \ j = 1, \dots, k.
$$
Thus, we get the desired matrix $A$.
\end{proof}

\section{Complex symmetric semigroups}\label{sec-sg}
In the section, we give a characterization of $\cj$-selfadjoint semigroups induced by unbounded operators $(C_{\psi_t,\varphi_t})_{t \geq 0}$. To that aim, we recall some terminologies about unbounded semigroups together with relevant reviews and then prove Theorem \ref{thm-cssg}.
\subsection{Terminologies} \label{subsec-sg}
The motivation to study (semi)groups of bounded operators, comes from various areas such as the study of linear differential equations $x'(t)=Ax(t)$, the operator-valued exponential functions, etc. The central part to the semigroups is disclosed in 1948 with the generation theorem by Hille and Yosida. Thanks to the contributions of schools, the theory has a certain accomplishment, which is presented in the monograph \cite{zbMATH01354832}. What restricts the applicability of classical semigroups is the requirement that operators must be bounded. Hughes \cite{zbMATH03560218} extended the theory of semigroups from bounded to unbounded context. We take a little time to recall Hughes's paper.
\begin{defn}[{\cite{zbMATH03560218}}]\label{defn-202306281646}
A family $(C(t))_{t\geq 0}$ of unbounded, linear operators acting on a Banach space $X$ is called a \emph{unbounded $C_0$-semigroup} if there exists $x\in X$ subject to the following conditions: 
\begin{enumerate}
\item[(A1)] $x\in\bigcap\limits_{t,s\geq 0}\text{dom}[C(t)C(s)]\setminus\{0\}$;
\item[(A2)] $C(t)C(s)x=C(t+s)x$ for every $t,s\geq 0$;
\item[(A3)] $C(\cdot)x$ is continuous on $(0,\infty)$, and $\lim\limits_{t\to 0^+}\|C(t)x-x\|=0.$
\end{enumerate}
For such $(C(t))_{t\geq 0}$, the symbol $D(C)$ denotes the set of elements $x \in X$ that satisfy axioms (A1), (A2) and (A3).
\end{defn}

\noindent 
From now on, for $\omega\in\R$ and $x\in D(C)$, we denote
\begin{equation*}
\|x\|_\omega:=\sup\limits_{t\geq 0}e^{-\omega t}\|C(t)x\| \quad 
 \text{ and } \quad \Sigma_\omega :=\{x\in D(C):\|x\|_\omega<\infty\}.
\end{equation*}

\begin{defn}[{\cite{zbMATH03560218}}]\label{defn202307162040}
Let $(C(t))_{t\geq 0}$ be a unbounded $C_0$-semigroup. 
For each $\omega \in \R$, the operator $G^{\omega}$ is defined in $(\Sigma_{\omega}, \|\cdot\|_{\omega})$ as follows:
\begin{enumerate}
\item The domain
$\text{dom}(G^\omega)$ is the set of all elements $x\in\Sigma_\omega$ such that $C(\cdot)x$ is differentiable at every $t>0$ and there exists $y\in\Sigma_\omega$ subject to $y=\lim\limits_{t\to 0^+}\frac{C(t)x-x}{t}$;

\item The operator $G^{\omega}x:=\lim\limits_{t\to 0^+}\frac{C(t)x-x}{t}$.
\end{enumerate}
The \emph{generator} $G$ of $(C(t))_{t\geq 0}$ is defined so that the restriction of $G$ to each $\Sigma_{\omega}$ is $G^{\omega}$.
\end{defn}

\begin{rem}
If $\omega_1\leq\omega_2$, then $\Sigma_{\omega_1}\subseteq\Sigma_{\omega_2}$ and $G^{\omega_1}\preceq G^{\omega_2}$.
\end{rem}

\begin{defn}[\cite{zbMATH01354832}]
The family $(C(t))_{t\geq 0}$ is called a \emph{$C_0$-semigroup in the bounded sense} if the operator $C(t)$ is bounded for every $t\geq 0$ and $D(C)=X$.  
\end{defn}

\begin{defn}[\cite{zbMATH06626207, zbMATH07175347}]
A $C_0$-semigroup $(C(t))_{t\geq 0}$ is called \emph{complex symmetric} if there exists a conjugation $\calc$ such that the operator $C(t)$ is $\calc$-selfadjoint for every $t\geq 0$. In this case, $(C(t))_{t\geq 0}$ is often called \emph{$\calc$-selfadjoint}.  
\end{defn}

In this section, we are interested in semigroups defined by semiflows and semicocycles.
\begin{defn}\label{defn-vt}
A family $(\varphi_t)_{t \geq 0}$ of nonconstant holomorphic self-mappings of $\C^d$ is called a \emph{semiflow} if the followings hold:
\begin{enumerate}
\item $\varphi_0(z)=z$ for all $z \in \C^d$;
\item the mapping $t\mapsto\varphi_t(\cdot)$ is continuous on $[0, \infty)$;
\item $\varphi_{t+s} = \varphi_t \circ \varphi_s$ for all $t, s \geq 0$.
\end{enumerate}
%Likewise, if $t$, $s \in \R$, then it is called a \emph{flow}.
\end{defn}
\begin{rem}\label{rem-diff-semi}
    It was proven in \cite[Theorem 5]{zbMATH00094996} that if $ (\varphi_t)_{t \geq 0}$ is a semiflow, then $\varphi_t(z)$ is differentiable in $t$.
\end{rem}

\begin{defn}\label{defn-xi}
A family $(\psi_t)_{t \geq 0}$ of nonzero entire functions on $\C^d$ is called a \emph{semicocycle} with respect to the semiflow $(\varphi_t)_{t \geq 0}$ if
\begin{enumerate}
\item $\psi_0(z) = 1$ for all $z \in \C^d$;
\item the mapping $t \mapsto \psi_t(z)$ is differentiable on $[0, \infty)$ for each $z \in \C^d$;
\item $\psi_{t+s} = \psi_t \cdot (\psi_s \circ \varphi_t)$ for all $t, s \geq 0$.
\end{enumerate}
\end{defn}

For a family $\Psi$ of nonzero entire functions $(\psi_t)_{t \geq 0}$ on $\C^d$ and a family $\Phi$ of nonconstant holomorphic self-mappings $(\varphi_t)_{t \geq 0}$ of $\C^d$, we denote by $(C_{\Psi, \Phi}(t))_{t \geq 0}$ or, briefly, $C_{\Psi, \Phi}$ the family of weighted composition operators $(C_{\psi_t,\varphi_t})_{t \geq 0}$. The family $C_{\Psi, \Phi}$ is studied in various spaces: \cite{zbMATH04024093} treats for Hardy spaces, \cite{zbMATH00933265} for Dirichlet space, \cite{zbMATH06626207, zbMATH07175347} for the one variable Fock space. The survey \cite{zbMATH01107617} is a good source for the recent results on the semigroups in the bounded sense.

\subsection{The $\cj$-selfadjoint semigroups $C_{\Psi, \Phi}$} \label{subsec-cssg}
With all preparation in place, we give a characterization of $\cj$-selfadjoint semigroups $C_{\Psi, \Phi}$, i.e.
we prove Theorem \ref{thm-cssg}.
\begin{proof}[The proof of Theorem \ref{thm-cssg}] For the reader's convenience, we divide the proof into several parts.

(1) Suppose that $(C_{\Psi,\Phi}(t))_{t \geq 0}$ is a $\cj$-selfadjoint semigroup. It follows from Theorem \ref{thm-cs}, that $C_{\psi_t, \varphi_t} = C_{\psi_t, \varphi_t, \max}$ and the symbols $\psi_t$ and $\varphi_t$ are of the forms \eqref{eq-cs} with the conditions \eqref{eq-cd-cs}, i.e.
$$
\psi_t(z) = \theta_t e^{\inner{z}{\ell_t}} \ \ \text{ and } \ \ \varphi_t(z) = Q_tz + q_t
$$
with
\begin{gather}
\label{eq-sym} 
(AQ_s)^{\bft}=AQ_s \quad \text{ and 
 } \quad \ell_t=\overline{Aq_t}+\overline{b}-Q_t^*\overline{b}.
\end{gather}
Since the family $(\varphi_t)_{t \geq 0}$ is a semiflow, we have
\begin{gather}
\label{eq-sf}   
Q_0=I, \qquad  
    q_0=0, \qquad Q_{t+s}=Q_tQ_s, \qquad q_{t+s}=Q_tq_s+q_t.
\end{gather} 
Using this and \cite[Chapter I,Theorem 2.9]{zbMATH01354832}, we have $Q_t = e^{t\Omega}$
, where $\Omega$ is a $d \times d$ matrix. Substituting this form back into \eqref{eq-sym}, we observe $e^{s\Omega^{\bft}}A^{\bft}=Ae^{s\Omega}$, which implies, after differentiating at $s=0$, that $\Omega^{\bft}A^{\bft}=A\Omega$, i.e. $(A\Omega)^{\bft}=A\Omega$.

As mentioned in Remark \ref{rem-diff-semi}, the function $\vap_t(z)$ is differentiable in $t$, and so is $q_t$. Since $q_0 = 0$, equation $q_{t+s} = Q_t q_s + q_t$ can be rewritten as follows
\begin{gather*}
    \frac{q_{t+s}-q_t}{s}=Q_t\left(\frac{q_s-q_0}{s}\right).
\end{gather*}
Letting $s\to 0^+$, the line above gives $q_t'=Q_tq_0'$, and so get 
$$
q_t = \int_0^t e^{s\Omega}q_{\diamond} ds.
$$

Since the family $(\psi_t)_{t \geq 0}$ is a semicocycle, by \eqref{eq-csg}, its coefficients satisfy
\begin{gather}
\label{eq-sco}
\ell_0=0, \qquad \tha_0=1,\qquad \tha_{t+s}=\tha_t\tha_s e^{\inner{q_t}{\ell_s}},\qquad 
\ell_{t+s}=Q_t^*\ell_s+\ell_t.
\end{gather}
Furthermore, by Definition \ref{defn-xi} the mapping $t \mapsto \psi_t(z)$ is differentiable $[0, \infty)$ for every $z \in \C^d$, and so $\tha_t,\ell_t$ are, too. Similarly as above, the equation $\ell_{t+s} = Q^*_t \ell_s + \ell_t$ under the condition $\ell_0 = 0$ gives
$$
\ell_t = \int_0^t e^{s\Omega^*}\ell_{\diamond}\; ds.
$$
Using this and differentiating the second equation in \eqref{eq-sym} two times at $t=0$, we get the second condition in  \eqref{eq-cd-csg}.

Next, using $\theta_0 = 1$, we can rewrite $\theta_{t+s} = \theta_t \theta_s e^{\langle q_t, \ell_s \rangle}$ as
\begin{gather*}
    \frac{\tha_{t+s}-\tha_t}{s}=\tha_t\left[e^{\inner{q_t}{\ell_s}}\frac{\tha_s-1}{s}+\frac{e^{\inner{q_t}{\ell_s}}-1}{s}\right].
\end{gather*}
Letting $s\to 0^+$, the line above gives $\tha_t'=\tha_t[\tha_0'+\inner{q_t}{\ell_0'}]$, and so get 
$$
\theta_t = e^{\theta_{\diamond} t + \int_0^t \langle q_s, \ell_{\diamond} \rangle \; ds}.
$$

(2) Conversely, take the symbols of the form \eqref{eq-csg} with coefficients of the forms \eqref{eq-co-csg} and the conditions \eqref{eq-cd-csg} so that $C_{\psi_t, \varphi_t} = C_{\psi_t, \varphi_t,\max}$ for every $t \geq 0$. Firstly, by \eqref{eq-cj} and induction on $n$, conditions \eqref{eq-cd-csg} give
\begin{gather*}
\left(\Omega^*\right)^n=\overline{A\Omega^n}A, \ \text{ i.e. }  \left(A \Omega^n\right)^{\bft}=A\Omega^n,\\  
\left(\Omega^*\right)^n\ell_\diamond=\overline{A\Omega^n q_\diamond}-(\Omega^*)^{n+1}\overline{b}.
\end{gather*}
Since the matrix exponential takes the form $e^{tZ}=\sum\limits_{n=0}^{\infty}\frac{t^nZ^n}{n!}$, equalities above yield conditions \eqref{eq-sym}, which together with \eqref{eq-csg}, by Theorem \ref{thm-cs} imply that  the operator $C_{\psi_t,\varphi_t,\max}$ is $\cj$-selfadjoint for every $t\geq 0$.

It remains to prove the family $C_{\Psi, \Phi}$ is a $C_0$-semigroup by showing that the function $K_z\in D(C_{\Psi,\Phi}).$ To do this, we use conditions \eqref{eq-sf} and \eqref{eq-sco} which follow from \eqref{eq-co-csg} and the fact that $K_z\in\text{dom}(C_{\psi_t,\varphi_t,\max})$ and, by \eqref{eq-ekz},
\begin{gather*}
    C_{\psi_t,\varphi_t,\max}K_z=\tha_te^{\inner{q_t}{z}}K_{Q_t^*z+\ell_t}.
\end{gather*}
Indeed, for every $t, s \geq 0$, using these conditions we get
\begin{align*}
 C_{\psi_t,\varphi_t,\max}C_{\psi_s,\varphi_s,\max}K_z &  = C_{\psi_t,\varphi_t,\max} \left(\tha_se^{\inner{q_s}{z}}K_{Q_s^*z+\ell_s}\right) \\
& =  \; \tha_s e^{\inner{q_s}{z}} \tha_t e^{\inner{q_t}{Q_s^*z+\ell_s}} K_{Q_t^*(Q_s^*z+\ell_s) + \ell_t} = C_{\psi_{t + s},\varphi_{t + s},\max}K_z,
\end{align*}
which implies axioms (A1) and (A2). For the remaining axiom, we consider
\begin{gather*}
    \|C_{\psi_t,\varphi_t,\max}K_z-C_{\psi_s,\varphi_s,\max}K_z\|^2
    =\|\tha_te^{\inner{q_t}{z}}K_{Q_t^*z+\ell_t}-\tha_se^{\inner{q_s}{z}}K_{Q_s^*z+\ell_s}\|^2\\
    =\|\tha_te^{\inner{q_t}{z}}K_{Q_t^*z+\ell_t}\|^2+\|\tha_se^{\inner{q_s}{z}}K_{Q_s^*z+\ell_s}\|^2
    -2\re\left(\tha_t\overline{\tha_s}e^{\inner{q_t}{z}+\inner{z}{q_s}}\inner{K_{Q_t^*z+\ell_t}}{K_{Q_s^*z+\ell_s}}\right)\\
    =|\tha_te^{\inner{q_t}{z}}|^2e^{|Q_t^*z+\ell_t|^2}+|\tha_se^{\inner{q_s}{z}}|^2e^{|Q_s^*z+\ell_s|^2}
    -2\re\left(\tha_t\overline{\tha_s}e^{\inner{q_t}{z}+\inner{z}{q_s}}K_{Q_t^*z+\ell_t}(Q_s^*z+\ell_s)\right).
\end{gather*}
Thus, letting $t\to s$, we observe that the function $K_z$ satisfies the axiom (A3).

(3) Now we study the generator of the $\cj$-selfadjoint semigroup $(C_{\Psi,\Phi}(t))_{t \geq 0}$.
We define the operator $W:\text{dom}(W)\subseteq\spc\to\spc$ by setting
\begin{gather*}
\text{dom}(W):=\bigcup_{\omega\in\R}\big\{h\in\Sigma_\omega:(\tha_\diamond+\inner{z}{\ell_\diamond})h(z)+\inner{\nabla h(z)}{\Omega z+q_\diamond}\in\Sigma_\omega \big \},\\
    Wh(z):=(\tha_\diamond+\inner{z}{\ell_\diamond})h(z)+\inner{\nabla h(z)}{\Omega z+q_\diamond},    
\end{gather*}
and denote by $W^{\omega}$ the restriction of $W$ to each $\Sigma_{\omega}$.

For $f\in\text{dom}(G^\omega)$, by Definition \ref{defn202307162040}, we have
\begin{gather*}
    \lim\limits_{t\to 0^+}\left\|\frac{C_{\psi_t,\varphi_t}f-f}{t}-G^\omega f\right\|=0,
\end{gather*}
which implies, as convergence in $\spc$-norm implies a point convergence, that
\begin{gather*}
    \lim\limits_{t\to 0^+}\left(\frac{C_{\psi_t,\varphi_t}f(z)-f(z)}{t}-G^\omega f(z)\right)=0\quad \text{ for each } z\in\C^d.
\end{gather*}
Consequently, by \eqref{eq-sf} and \eqref{eq-sco}, 
\begin{gather*}
    G^\omega f(z)=(C_{\psi,\varphi}(t)f(z))'|_{t=0}=(\tha_\diamond+\inner{z}{\ell_\diamond})f(z)+\inner{\nabla f(z)}{\Omega z+q_\diamond},
\end{gather*}
which gives $G^\omega\preceq W^\omega$. 

To prove the inverse inclusion, by using of \cite[Lemma 1.3]{zbMATH06046473} it is sufficient to prove that the operator $G^\omega - \delta I$ is onto from $\text{dom}(G^\omega)$ to $\spc$ and the operator $W^\omega - \delta I$ is one-to-one from $\text{dom}(W^\omega)$ to $\spc$ for some number $\delta$. To do this, we denote
\begin{gather*}
    \mathscr D_\omega:=\{f\in\Sigma_\omega:\|C_{\psi_t,\varphi_t}f-f\|_\omega\to 0\,\ \text{ as } \ t\to 0^+\}.
\end{gather*}
It was proven in \cite[Theorem 2.20]{zbMATH03560218} that the restriction of $(C_{\psi_t,\varphi_t})_{t \geq 0}$ on $\mathscr D_\omega$ is a semigroup in the bounded sense. By \cite[Proposition I.5.5, Generation Theorem II.3.8]{zbMATH01354832}, we have
\begin{gather*}
\|C_{\psi_t,\varphi_t}f\|_\omega\leq pe^{t\delta_0}\|f\|_\omega\quad \text{ for all } t\geq 0 \text{ and }  f\in\mathscr{D}_\omega
\end{gather*}
and $(\delta_0,+\infty)\subset\rho(G^\omega)$, where $\rho(G^\omega)$ is the resolvent of $G^\omega$. Let $\delta > \delta_0.$ It follows from \cite[Chapter II, Theorem 1.10]{zbMATH01354832} that $\delta\in\rho(G^\omega)$; meaning that the operator $G^\omega-\delta I$ is onto from $\text{dom}(G^\omega)$ to $\spc$. Next, we prove that $W^\omega-\delta I$ is one-to-one from $\text{dom}(W^\omega)$ to $\spc$. Let $f\in\text{dom}(W^\omega)$ such that
\begin{gather*}
    0=(W^\omega-\delta I)f(z)=(\tha_\diamond+\inner{z}{\ell_\diamond}-\delta)f(z)+\inner{\nabla f(z)}{\Omega z+q_\diamond}.
\end{gather*}
Letting $z=\vap_t(x)=Q_tx+q_t$, where $x\in\C^d$ is chosen arbitrarily, the line above becomes
\begin{gather}\label{202310102112}
    (\tha_\diamond-\delta+\inner{\vap_t(x)}{\ell_\diamond})f(\vap_t(x))+\inner{\nabla f(\vap_t(x))}{\Omega Q_tx+\Omega q_t+q_\diamond}=0.
\end{gather}
Note that
\begin{gather}\label{202310102111}
    \frac{d}{dt}[\vap_t(x)]=\Omega Q_tx+e^{t\Omega}q_\diamond
    %=\Omega Q_tx+\Omega\int\limits_0^te^{s\Omega}q_\diamond\,ds+q_\diamond
    =\Omega Q_tx+\Omega q_t+q_\diamond,
\end{gather}
where the second equality uses \cite[Chapter II, Lemma 1.3]{zbMATH01354832}. Through substituting \eqref{202310102111} back into \eqref{202310102112} and then setting $v(t)=f(\vap_t(x))$, we get
\begin{gather*}
(\tha_\diamond-\delta+\inner{\vap_t(x)}{\ell_\diamond})v(t)+v'(t)=0,
\end{gather*}
which has the solution
\begin{gather*}
    v(t)=v(0)\exp\left(-\int\limits_0^t(\tha_\diamond-\delta+\inner{\vap_s(x)}{\ell_\diamond})\,ds\right).
\end{gather*}
Consequently, taking into account the form of $v(t)$ and using \eqref{eq-co-csg}, we have
\begin{align*}
f(\vap_t(x)) & =v(t)=f(x)\exp\left((\delta-\tha_\diamond)t-\int\limits_0^t\inner{Q_sx+q_s}{\ell_\diamond}\,ds\right)\\
& =f(x)\frac{1}{\tha_t}\exp\left(\delta t-\int\limits_0^t\inner{Q_sx}{\ell_\diamond}\,ds\right)
=f(x)\frac{1}{\tha_t}\exp\left(\delta t-\inner{x}{\ell_t}\right),
\end{align*}
which by \eqref{eq-csg} yields
\begin{gather*}
\psi_t(x)f(\vap_t(x))=e^{\delta t}f(x)\in\spc.    
\end{gather*}
Since the domain $\text{dom}(C_{\psi_t,\varphi_t})$ is maximal, we infer $f\in\text{dom}(C_{\psi_t,\varphi_t})$ and furthermore
\begin{gather*}
C_{\psi_t,\varphi_t}f(x)=\psi_t(x)f(\vap_t(x))=e^{\delta t}f(x).
\end{gather*}
The line above gives $f\in\mathscr{D}_\omega$. Hence,
\begin{gather*}
    e^{\delta t}\|f\|_\omega=\|C_{\psi_t,\varphi_t}f\|_\omega\leq e^{\delta_0 t}\|f\|_\omega.
\end{gather*}
Since $\delta>\delta_0$, the inequality above yields $f\equiv 0$.

Finally, the $\cj$-symmetry follows directly from \cite[Theorem 7.11]{zbMATH07175347}.
\end{proof}

%\subsection{The bounded case}
%Now we focus only on semigroups that consists of bounded operators. To that aim, assume that
%\begin{gather}\label{asu-20231112}
    %\text{For every $t\geq 0$, the operator $C_{\psi,\varphi}(t)$ is bounded.}
%\end{gather}

\section*{Acknowledgement}
This research is funded by Vietnam
National Foundation for Science and Technology Development (NAFOSTED)
under grant number 101.02-2021.24

\nocite{*}
\bibliographystyle{plain}
\bibliography{refs}
\end{document}